\documentclass[leqno,11pt]{amsart}
\usepackage{amsmath,amstext,amssymb,amsopn,amsthm,mathrsfs,wasysym,dsfont,comment}

\theoremstyle{plain}

\allowdisplaybreaks

\newtheorem{thm}{Theorem}[section]
\newtheorem{cor}[thm]{Corollary}

\newtheorem{ob}[thm]{Observation}

\newtheorem{lem}[thm]{Lemma}

\newtheorem*{ex*}{Example}
\newtheorem{bigthm}{Theorem}
\newtheorem{bigcor}[bigthm]{Corollary}

\newcommand{\N}{\mathbb{N}}
\newcommand{\C}{\mathbb{C}}
\newcommand{\R}{\mathbb{R}}
\newcommand{\Z}{\mathbb{Z}}

\def\H{(-D)}

\def\D{{D}}
\def\L{{L}}

\def\1{\boldsymbol{1}}

\def\sv{\sin \varphi}

\def\vt{\vartheta}
\def\8{\infty}

\def\8{\infty}

\def\vp{\varphi}

\DeclareMathOperator{\dist}{dist}

\title[Spherical heat kernel]
	{Sharp estimates of the spherical heat kernel}

\author[A. Nowak]{Adam Nowak}
\address{Adam Nowak, \newline
			Institute of Mathematics,
		Polish Academy of Sciences, \newline
      \'Sniadeckich 8,
      00--656 Warszawa, Poland    
      }
\email{anowak@impan.pl}

\author[P. Sj\"ogren]{Peter Sj\"ogren}
\address{Peter Sj\"ogren, \newline
			Mathematical Sciences, University of Gothenburg \newline
Mathematical Sciences, Chalmers University of Technology \newline
SE-412 96 G\"oteborg, Sweden 
      }
\email{peters@chalmers.se}

\author[T.Z. Szarek]{Tomasz Z. Szarek}
\address{Tomasz Z. Szarek,     \newline
            Institute of Mathematics,
        Polish Academy of Sciences, \newline
      \'Sniadeckich 8,
      00--656 Warszawa, Poland \newline
\indent and \newline
            Department of Mathematics, Informatics and Mechanics,
        University of Warsaw, \newline
        Banacha 2,
        02--097 Warszawa, Poland
      }
\email{szarek@impan.pl}

\begin{document}

\begin{abstract}
We prove sharp two-sided global estimates for the heat kernel associated with a Euclidean sphere of
arbitrary dimension. 
\end{abstract}

\maketitle

\footnotetext{
\indent	
The first and the third authors were supported by the National Science Centre of Poland within the research project OPUS 2017/27/B/ST1/01623. The third author was supported also by the Foundation for Polish Science via the START Scholarship.
}

\section{Statement of the result} \label{sec:res}

Let $S^d \subset \R^{d+1}$ be the Euclidean unit sphere of dimension $d \ge 1$ equipped with the standard non-normalized
area measure $\sigma_d$. The heat kernel $\mathcal{K}_t^d(\xi,\eta)$ on $(S^d,\sigma_d)$ is a function of the
geodesic spherical distance $\dist(\xi,\eta) = \arccos\langle\xi,\eta\rangle$, and we write it as $K^d_t(\cdot)$, i.e.,
$$
\mathcal{K}_t^d(\xi,\eta) = K_t^d\big(\dist(\xi,\eta)\big), \qquad \xi,\eta \in S^d.
$$
In this paper we prove the following.
\begin{bigthm} \label{thm:main}
Let $d \ge 1$ and $T>0$ be fixed. For all $\vp \in [0,\pi]$ and $0 < t \le T$ 
\begin{equation*}
\frac{c}{(t + \pi - \vp)^{(d-1)/2} t^{d/2}} \exp\bigg({-\frac{\vp^2}{4t}}\bigg) \le
K_t^d (\vp) \le \frac{C}{(t + \pi - \vp)^{(d-1)/2} t^{d/2}} \exp\bigg({-\frac{\vp^2}{4t}}\bigg)
\end{equation*}
holds with some constants $c,C>0$ depending only on $d$ and $T$.
\end{bigthm}
Analogous sharp bounds of $K_t^d(\vp)$ for large $t$ are well known; one has
\begin{equation*}
c \le K_t^{d}(\vp) \le C, \qquad \vp \in [0,\pi], \quad t \ge T,
\end{equation*}
for any fixed $T>0$. This is also a consequence of our estimates for small $t$ together with the semigroup property.

Theorem \ref{thm:main} leads to sharp bounds for the derivative $\partial_{\vp}K_t^d(\vp)$.
We have the following result which, in particular, confirms the intuitively obvious fact that $K_t^d$ is strictly
decreasing in $[0,\pi]$.
\begin{bigcor} \label{cor:main2}
Let $d \ge 1$ and $T>0$ be fixed. There exist constants $c,C>0$ depending only on $d$ and $T$ such that
for $\vp\in [0,\pi]$ and $0<t\le T$
\begin{equation*}
\frac{c \,\vp(\pi-\vp)}{(t + \pi - \vp)^{(d+1)/2} t^{d/2+1}} \exp\bigg({-\frac{\vp^2}{4t}}\bigg) \le
-\partial_{\vp}K_t^d (\vp) \le \frac{C \,\vp(\pi-\vp)}{(t + \pi - \vp)^{(d+1)/2} t^{d/2+1}} \exp\bigg({-\frac{\vp^2}{4t}}\bigg),
\end{equation*}
while for $t \ge T$
$$
c\, e^{-td} \,\vp (\pi-\vp) \le - \partial_{\vp}K_t^d(\vp) \le C\, e^{-td} \vp (\pi-\vp).
$$
\end{bigcor}

The spherical heat kernel is an important object in analysis, probability and physics, among other fields.
It is the integral kernel of the spherical heat semigroup and thus provides solutions to the heat equation based on the
Laplace-Beltrami operator on $S^d$. It is also a transition probability density of the spherical Brownian motion.
Clearly, these two facts lead to physical significance and applications.

Surprisingly enough, up to our best knowledge an exact global description of the decisive exponential behavior
of $K_t^d(\vp)$ for small $t$ has not been established before, except for the simple case $d=1$ in which the kernel
is just a periodization of the Gauss-Weierstrass kernel.
The main reason and obstacle seems to be the geometry of the sphere that has to be taken into account,
but technically is not easy to handle.
Indeed, intuitively it is clear that the behavior of $K_t^{d}(\vp)$ is very different for small $\vp$, where the sphere
resembles $\R^{d}$, and close to the antipodal point $\vp = \pi$, to which, roughly speaking, the heat can flow
along many geodesic lines.

The most precise global bounds for $K_t^d(\vp)$ known so far are only qualitatively sharp.
By this we mean that the number $4$
in the exponential factors in Theorem~\ref{thm:main} is replaced by some smaller and larger numbers in the lower and
upper bounds, respectively; see e.g.\ Theorems 5.5.6 and 5.6.1 in \cite{Da}.
A sharp estimate for the antipodal point $K_t^{d}(\pi)$ was found by Molchanov, see \cite[Example 3.1]{M}.
For dimensions $d=2,3$ some partial results in the spirit of Theorem \ref{thm:main}, in particular the upper bound,
were obtained by Andersson~\cite{And}.
In this context it is perhaps interesting to note that Nagase \cite[Theorem 1.1]{N}
found a very precise description of the asymptotic behavior of $K_t^d(\vp)$ as $t \to 0$, for small values of $\vp$.

In contrast to qualitatively sharp estimates, genuinely sharp heat kernel bounds are usually much harder to prove and
appear rarely in the literature; heat kernel estimates on the hyperbolic space \cite{DaMa} is one of these sparse instances.
The example of $S^d$ shows that this is a difficult problem even for basic and regular Riemannian manifolds.
In this connection, it is perhaps worth mentioning the recent papers
\cite{BM1,BM2,MS,MSZ} where such results were obtained for Dirichlet heat kernels related to Bessel operators in half-lines,
the Dirichlet heat kernel in Euclidean balls of arbitrary dimension, and the Fourier-Bessel heat kernel on the interval $(0,1)$.
This was achieved by a clever combination of probabilistic and analytic methods.

An interesting aspect of Theorem \ref{thm:main} is its relation with sharp estimates for the ultraspherical, or more generally,
the Jacobi heat kernel $G_t^{\alpha,\beta}(x,y)$; see e.g.\ \cite{NoSj}. Qualitatively sharp estimates for the Jacobi heat kernel
were obtained independently in \cite{CKP} and \cite{NoSj}. Combining our Theorem \ref{thm:main} with the reduction formula
derived in \cite{NoSj}, one can prove genuinely sharp bounds for $G_t^{\alpha,\beta}(x,y)$, assuming that
$\alpha,\beta \ge -1/2$ and $\alpha+\beta$ is a dyadic number. This leads to the natural conjecture that
\cite[Theorem A]{NoSj} holds with $c_1=c_2=1/4$, that is exactly the same constants in the exponential factors as in
Theorem~\ref{thm:main}, and this for all $\alpha,\beta > -1$.

\vspace{0,5cm}
\section{Outline of the proof}

The spherical heat kernel $\mathcal{K}_t^d(\xi,\eta)$ and the associated kernel $K_t^d(\vp)$
can be expressed explicitly as series involving spherical harmonics or ultraspherical polynomials, respectively.
But these series oscillate heavily and in general cannot be computed, so they are of no use for our purposes.
Thus our approach is less direct.

In the first step, we prove Theorem \ref{thm:main} for odd dimensions $d=1,3,5,\ldots$. This is done by exploiting in
an elementary, though technically involved, way the recurrence relation
\begin{equation} \label{rec_main}
K_t^{d+2} (\vp) 
=
 - \frac{e^{td}}{2\pi} 
(\sv)^{-1}
\partial_{\vp} K_t^{d}(\vp), 
\qquad d \ge 1, 
\end{equation}
together with the well-known fact that $K_t^1(\vp)$ is given by a simple positive series
\begin{equation} \label{3theta}
K_t^1 (\vp) 
=
\vt_t(\vp) :=
\sum_{n \in \Z} W_t(\vp + 2\pi n).
\end{equation}
Here $W_t$ is the one-dimensional Gauss-Weierstrass kernel
$$
W_t(x)  = \frac{1}{\sqrt{4 \pi t}} e^{-x^2/(4t)}.
$$
It is worth mentioning that $\vt_t(\vp)$ can be expressed in terms of $\theta_3$, one of the celebrated Jacobi theta functions.
Notice that \eqref{rec_main} readily implies Corollary \ref{cor:main2} once Theorem \ref{thm:main} is proved.

The formula \eqref{rec_main} is a special case of a more general relation satisfied by the ultraspherical
heat kernel
\begin{equation} \label{rec_Jac}
\partial_x G_t^{\alpha,\alpha}(x,1) = 2 (\alpha+1) e^{-t(2\alpha+2)} G_t^{\alpha+1,\alpha+1}(x,1), \qquad \alpha > -1,
\end{equation}
since
\begin{equation} \label{rel_Jac}
K_t^d(\vp) = \frac{1}{\sigma_{d-1}(S^{d-1})} G_t^{d/2-1,d/2-1}(\cos\vp,1), \qquad d \ge 1,
\end{equation}
where $\sigma_0(S^0)=2$ in case $d=1$. The identity \eqref{rec_Jac} follows by a straightforward
differentiation of the series expressing $G_t^{\alpha,\alpha}(x,1)$ in terms of ultraspherical polynomials.
Both \eqref{rec_Jac} and \eqref{rec_main} can be found e.g.\ in \cite[(2.7.13)]{BGL} and \cite[(2.7.15)]{BGL}, respectively,
but they were no doubt known earlier, at least as folklore.

In the second step we use the result for odd dimensions to cover all even dimensions $d=2,4,6,\ldots$.
This is performed by employing a reduction formula for the Jacobi heat kernel \cite[Theorem 3.1]{NoSj} that in our situation
implies via \eqref{rel_Jac}
\begin{equation} \label{reduc1}
K_t^d(\vp) = c_d \int_{-1}^1 K_{t/4}^{2d-1}\bigg(\arccos\Big(v\cos\frac{\vp}2\Big)\bigg) \big(1-v^2\big)^{(d-3)/2}\, dv,
	\qquad d \ge 2,
\end{equation}
with $c_d = 2^{-d+1} \pi^{(d-1)/2}/\Gamma((d-1)/2)$.
It is worth noting that this step of the proof can be generalized to deliver
an analogue of Theorem \ref{thm:main} in the Jacobi setting, 
as mentioned in Section~\ref{sec:res}.

Summing up, we split the proof of Theorem \ref{thm:main} into the following two results.
\begin{thm} \label{thm:step1}
Given $N \ge 0$,
the estimate of Theorem \ref{thm:main} holds for $d=2N+1$. 
\end{thm}

\begin{thm} \label{thm:step2}
Let $N \ge 2$.
If the estimate of Theorem \ref{thm:main} holds in dimension $d=2N-1$, then it also holds in dimension $d=N$.
\end{thm}

The proofs of Theorems \ref{thm:step1} and \ref{thm:step2} are given in Sections \ref{sec:st1} and \ref{sec:st2}.
In both cases it is enough to show the result for some $T$, possibly very small, and one can restrict $\vp$ to
$(0,\pi)$, since $K_t^d(\vp)$ is continuous on the closed interval $[0,\pi]$.
Some preparatory results needed to prove Theorem~\ref{thm:step1} are contained in Section~\ref{sec:auxres}.

\subsection*{Notation}
In what follows
we denote by $\N = \{0,1,\ldots\}$ the set of natural numbers. We write $x \wedge y$ for the minimum of $x$ and $y$.
Further, we will frequently use the notation $X \lesssim Y$ to indicate that
$X \le C Y$ with a positive constant $C$ independent of significant quantities. We shall write
$X \simeq Y$ when simultaneously $X \lesssim Y$ and $Y \lesssim X$.

\vspace{0,5cm}
\section{Technical preparation} \label{sec:auxres}

Define differential operators
\begin{align*}
\D 
=
\frac{1}{\sin z} \frac{\mathrm{d}}{\mathrm{d}z}, \qquad
L 
=
\frac{1}{z} \frac{\mathrm{d}}{\mathrm{d}z}.
\end{align*}
Observe that $\L$ preserves the space of even, entire functions. Let $N \in \N$.
When writing $\D$, we will often need to specify the variable, and for instance
the notation $\D_z^N \big( F(v z) \big)$ will mean that the operator $\D^N$ is applied 
to the function $z \mapsto F(vz)$. On the other hand, we write $L^N F(vx)$ for $L^N F$ evaluated at $vx$.
Denote $E = \pi \Z \setminus \{0\}$.

\begin{lem}\label{lem:A}
Let $F$ be an even, entire function. Then for $N \ge 1$ and $v>0$
\begin{align} \label{for1} 
\D_z^N \big( F(v z) \big) 
=
\sum_{j=1}^N v^{2j} L^j F(vz) \Phi_{N,j} (z), 
\qquad z \in \C \setminus E.
\end{align}
Here the functions $\Phi_{N,j}$, $j=1,\ldots,N$, are even and meromorphic in $\C$ with poles only at the points of $E$,
and these poles are of order at most $2N-j$. Moreover, $\Phi_{N,N} (z) = ( z/\sin z )^N$.  
\end{lem}

\begin{proof}
When $N=1$, \eqref{for1} is obvious, since
$$
\frac{1}{\sin z} \frac{\mathrm{d}}{\mathrm{d}z} \big(F(vz)\big) = v^2 \frac{F'(vz)}{vz} \frac{z}{\sin z}.
$$
For the induction step from $N$ to $N+1$, we apply $D_z$ to term number $j$ in \eqref{for1} and use 
Leibniz' rule and the case $N=1$ to evaluate $D_z( L^j F(vz))$. The result is
\begin{equation} 
D_z\big( v^{2j} L^{j} F(vz) \Phi_{N,j}(z) \big) 
 = v^{2j+2} L^{j+1} F(vz) \frac{z}{\sin z} \Phi_{N,j}(z) + v^{2j} L^{j}F(vz) \frac{1}{\sin z}\Phi'_{N,j}(z).
\label{for9}
\end{equation}
Since $\Phi_{N,j}$ is even and analytic in $\C\setminus E$, so is the function $z \mapsto (\sin z)^{-1}\Phi'_{N,j}(z)$
appearing here, and its poles are of order at most $2+2N-j = 2(N+1)-j$. This means that the second term on the right-hand
side of \eqref{for9} fits in the sum in \eqref{for1} with $N$ replaced by $N+1$. The same is true for the first term
on the right-hand side of \eqref{for9}, since $z \mapsto z (\sin z)^{-1}\Phi_{N,j}(z)$ has poles of order at most
$1+ 2N-j = 2(N+1)-(j+1)$ in $E$.
This completes the induction step, because it is easy to verify, also by induction, that
$\Phi_{N,N}(z) = (z/\sin z)^N$. Lemma \ref{lem:A} is proved.
\end{proof}

\begin{lem}\label{lem:F4rr}
Let $j \ge 1$ and $M_0>0$ be fixed. Then
\begin{itemize}
\item[(a)] 
$\L^j(\cosh) (z) \simeq 1$ uniformly in $z \in (0,M_0]$ and
\item[(b)]
$| \L^j(\cosh) (z) | \lesssim e^z$ for $z > 0$.
\end{itemize}
\end{lem}

\begin{proof} 
Since $\cosh z = \sum_{k=0}^{\8} a_k z^{2k}$ with $a_k > 0$,
we see that $\L^j(\cosh) (z)$ will be of the same form, with coefficients $a_{k,j} > 0$.
Thus $\L^j(\cosh) (z) \ge a_{0,j} > 0$, $z > 0$, and this implies item (a).

To show (b), we may assume that $z \ge 1$ because of (a). By induction $\L^j(\cosh) (z)$ can be seen to be a finite
linear combination of terms $z^{-m}\sinh z$ and $z^{-m}\cosh z$ with $m \ge 1$.
So $|\L^j(\cosh) (z) | \lesssim e^z$ for $z \ge 1$.
\end{proof}

\begin{lem} \label{lem:2}
Let $k \in \N$ be fixed. Then, for any $M_0 > 0$ there exists $v_0 > 0$ such that 
\begin{align*} 
\D_z^k \big( \cosh (v z) \big) 
\simeq
v^{2k},
\end{align*}
uniformly in $z \in (0,\pi/2]$ and $v \ge v_0$ satisfying $vz \le M_0$.
\end{lem}

\begin{proof}
The case $k=0$ is trivial, so we consider $k \ge 1$. 
Using Lemma~\ref{lem:A} with $F = \cosh$, we see that for any fixed $k \ge 1$
\begin{align} \label{iden100n} 
\D_z^k \big( \cosh (v z) \big) 
=
\sum_{j=1}^k v^{2j} L^j (\cosh) (vz) \Phi_{k,j} (z), 
\end{align}
and 
\begin{align} \label{estPhin}
\Phi_{k,k} (z) \simeq (\pi - z)^{-k}, 
\qquad
| \Phi_{k,j} (z) |
\lesssim
(\pi - z)^{-2k+j},
\qquad z \in (0,\pi)
\end{align}
for $1 \le j < k$.
Combining this with Lemma~\ref{lem:F4rr} (a), we see that
the $k$th term of the sum in \eqref{iden100n} dominates if $v$ is large enough. This implies the lemma.
\end{proof}

We will also need the following modification of Lemma~\ref{lem:2}.
\begin{lem} \label{lem:ch_gl}
Let $k \in \N$ be fixed. Then
\begin{align*} 
\big| \D_z^k \big( \cosh (v z) \big)  \big|
\lesssim
e^{vz} v^{2k} (\pi - z)^{-2k},
\qquad v \ge 1, \quad z \in (0,\pi).
\end{align*}
\end{lem}

\begin{proof}
The case $k=0$ is trivial, so let $k \ge 1$. 
Using \eqref{iden100n}, \eqref{estPhin} and Lemma~\ref{lem:F4rr} (b) we infer that
\begin{align*} 
\big| \D_z^k \big( \cosh (v z) \big)  \big|
\lesssim
\sum_{j=1}^k v^{2j}
e^{vz} (\pi - z)^{-2k+j}
\lesssim
e^{vz} v^{2k} (\pi - z)^{-2k},
\end{align*}
uniformly in $v \ge 1$ and $z \in (0,\pi)$, as desired.
\end{proof}

\vspace{0,5cm}
\section{Proof of Theorem \ref{thm:step1}} \label{sec:st1}

In order to prove Theorem~\ref{thm:step1}, it is enough to show the following, see \eqref{rec_main} and \eqref{3theta}.

\begin{thm} \label{pro:Kodd}
Let $N \in \N$ be fixed. There exists $t_0 > 0$ such that
\begin{align*}
\H^N \vt_t (\vp)
& \simeq 
W_t(\vp) (t + \pi - \vp)^{-N} t^{-N},
\qquad \vp \in (0,\pi), \quad 0 < t \le t_0.
\end{align*}
\end{thm}
This result is a straightforward consequence of the following two lemmas.

\begin{lem} \label{lem:thetalar}
Let $N \in \N$ be fixed. There exist $M_0, t_0 > 0$ and constants $c, C > 0$ such that
\begin{align*}
c
W_t(\vp) (\pi - \vp)^{-N} t^{-N}
\le 
\H^N \vt_t (\vp)
\le C
W_t(\vp) (\pi - \vp)^{-N} t^{-N}
\end{align*}
holds for $\vp \in (0,\pi)$ and $0 < t \le t_0$ satisfying 
$(\pi - \vp)/t \ge M_0$.
\end{lem}

\begin{lem} \label{lem:thetasm}
Let $N \in \N$ be fixed. For any $M_0 > 0$ there exist $t_0 > 0$ and constants $c, C > 0$ such that
\begin{align*}
c
W_t(\vp) t^{-2N}
\le 
\H^N \vt_t (\vp)
\le C
W_t(\vp) t^{-2N}
\end{align*}
holds for $\vp \in (0,\pi)$ and $0 < t \le t_0$ satisfying 
$(\pi - \vp)/t \le M_0$.
\end{lem}

The proofs of Lemmas~\ref{lem:thetalar} and \ref{lem:thetasm} will be given in Sections~\ref{ssec:1} and \ref{ssec:7}. 
First, however, we need two crucial intermediate results, Lemmas \ref{lem:ularge} and \ref{lem:usmall} below.

\begin{lem} \label{lem:ularge}
The assertion of Lemma \ref{lem:thetalar} is true if $\vt_t$ is replaced by $W_t$, and this with $t_0=\infty$.
\end{lem}

Observe that if Lemma \ref{lem:thetalar} (hence also Lemma \ref{lem:ularge})
holds with some $M_0 > 0$, then it also holds with any larger $M_0$ and the same $t_0,c,C$.
This will be frequently used in the sequel without further mention. 
Furthermore, we have the following straightforward consequence of Lemma~\ref{lem:ularge}, which will be needed in the proof of Lemma~\ref{lem:usmall}.

\begin{cor} \label{cor:ularge}
Let $N \in \N$ be fixed. There exists $t_0 > 0$ such that
\begin{align*}
\H^N W_t (\psi)
\simeq
W_t(\psi) t^{-N},
\qquad \psi \in (0,\pi/2], \quad 0 < t \le t_0.
\end{align*}
\end{cor}

\begin{lem} \label{lem:usmall}
The assertion of Lemma \ref{lem:thetasm} is true if $\vt_t$ is replaced by $W_t+W_t(\cdot - 2\pi)$.
\end{lem}

Lemmas \ref{lem:ularge} and \ref{lem:usmall} are proved in Sections \ref{ssec:ularge} and \ref{ssec:usmall}, respectively.

\subsection{Proof of Lemma \ref{lem:ularge}} \label{ssec:ularge}

For $N=0$ there is nothing to prove, so let $N \ge 1$.
Applying Lemma~\ref{lem:A} with $F = W_1$ and $v = t^{-1/2}$ and using the identity $L^j W_1 = (-1)^j 2^{-j} W_1$, we get
\begin{align} \label{iden111}
\H^N W_t (\vp)
=
\sum_{j=1}^N \frac{(-1)^{N+j}}{2^{j}} t^{-j} W_t (\vp) \Phi_{N,j} (\vp). 
\end{align}
Since $\Phi_{N,N} (\vp) \simeq (\pi - \vp)^{-N}$, $\vp \in (0,\pi)$, term number $N$ in the above sum is comparable to
$W_t(\vp) (\pi - \vp)^{-N} t^{-N}$ for $\vp \in (0,\pi)$ and $t>0$. Using the bounds 
$| \Phi_{N,j} (\vp) | \lesssim (\pi - \vp)^{-2N+j}$, $\vp \in (0,\pi)$ for $1 \le j < N$,
we see that the remaining terms are controlled by
\begin{align*} 
W_t (\vp)
\sum_{j=1}^{N-1} t^{-j} (\pi - \vp)^{-2N+j}
& =
W_t (\vp) (\pi - \vp)^{-N} t^{-N} \sum_{j=1}^{N-1}
\Big( \frac{t}{\pi - \vp} \Big)^{N - j},
\end{align*}
uniformly in $\vp \in (0,\pi)$ and $t > 0$. 
Choosing $M_0$ large enough, we can make term number $N$ dominate, and Lemma~\ref{lem:ularge} follows.
\qed

\subsection{Proof of Lemma \ref{lem:usmall}} \label{ssec:usmall}
For $N=0$ the conclusion of Lemma~\ref{lem:usmall} is straightforward, so assume $N \ge 1$. 
Letting $\psi = \pi - \vp$, we have $\D_\psi = - \D_\vp$ and 
\begin{align*} 
W_t (\vp) + W_t (\vp - 2\pi)
=
W_t ( \pi - \psi ) + W_t ( \pi + \psi ) 
=
2 e^{-\pi^2/(4t)} W_t (\psi) \cosh \frac{\pi \psi}{2t}.
\end{align*} 
Thus
\begin{align} \label{iden20}
(-\D)^N \big[ W_t(\vp) + W_t(\vp - 2\pi) \big] 
=
2 e^{-\pi^2/(4t)}
\sum_{k=0}^N \binom{N}{k} \D^{N-k} W_t (\psi) 
\D_\psi^k \Big( \cosh \frac{\pi \psi}{2t} \Big). 
\end{align}
Here $\psi/t \le M_0$, and by choosing $t_0$ small we can assume that $0 < \psi \le \pi/2$.
Then Corollary~\ref{cor:ularge} implies that for some $t_0 > 0$ and all $0 \le k < N$
\begin{align*} 
| \D^{N-k} W_t (\psi) |
\simeq
W_t (\psi) t^{-(N-k)}, \qquad \psi \in (0, \pi/2], \quad 0 < t \le t_0.
\end{align*}
Applying Lemma~\ref{lem:2} (taken with $\pi M_0/2$ instead of $M_0$ and $v = \pi/(2t)$) and making $t_0$ smaller if necessary, we get
\begin{align*} 
\D_\psi^k \bigg( \cosh \frac{\pi \psi}{2t} \bigg) 
\simeq 
t^{-2k}, \qquad 0 \le k \le N,
\end{align*}
uniformly in $\psi \in (0,\pi/2]$ and $0 < t \le t_0$ satisfying 
$\psi/t \le M_0$.

The term with $k=N$ in the sum in \eqref{iden20} is 
\begin{align*} 
2 e^{-\pi^2/(4t)}
W_t (\psi) 
\D_\psi^N \Big( \cosh \frac{\pi \psi}{2t} \Big)
\simeq
e^{-\pi^2/(4t)}
W_t (\psi) t^{-2N}.
\end{align*}
The terms with $k<N$ in this sum can be made much smaller as $t$ approaches to $0$,
since they are controlled by $e^{-\pi^2/(4t)}
W_t (\psi) t^{-N-k}$. It follows that for sufficiently small $t_0>0$ 
\begin{align*} 
(-\D)^N \big[ W_t(\vp) + W_t(\vp - 2\pi) \big] 
\simeq
e^{-\pi^2/(4t)}
W_t (\psi) t^{-2N},
\end{align*}
uniformly in $\psi \in (0,\pi/2]$ and $0 < t \le t_0$ satisfying 
$\psi/t \le M_0$.
Finally, using the relation $e^{-\pi^2/(4t)} W_t (\psi) \simeq W_t (\vp)$ for $\psi/t \le M_0$, we conclude the proof.
\qed

\subsection{Proof of Lemma \ref{lem:thetalar}} \label{ssec:1}

The case $N=0$ is straightforward, so assume that $N \ge 1$. We will show that there exists an $M_1 \ge 1$ such that
\begin{align} \label{est1}
 \sum_{n=1}^{\8} 
\big| \D^N \big[ W_t(\vp - 2\pi n) + W_t(\vp + 2\pi n) \big]  \big| 
\lesssim
W_t(\vp) (\pi - \vp)^{-N} t^{-N} \big( e^{-\pi^2/(4t)} + e^{-\pi (\pi - \vp)/t} \big),
\end{align}
uniformly in $\vp \in (0,\pi)$ and $0 < t \le 1$ satisfying 
$(\pi - \vp)/t \ge M_1$.

First observe that
\begin{align*}
W_t(\vp - 2\pi n) + W_t(\vp + 2\pi n) 
=
2 e^{-\pi^2n^2/t} W_t(\vp) \cosh \frac{\pi n \vp}{t},
\qquad \vp \in (0,\pi), \quad t > 0, \quad n \ge 1.
\end{align*}
Consequently, using Leibniz' rule for $\D$ we obtain
\begin{align*}
\D^N \big[ W_t(\vp - 2\pi n) + W_t(\vp + 2\pi n) \big] 
=
2 e^{-\pi^2n^2/t} \sum_{k=0}^N \binom{N}{k} \D^{N-k} W_t (\vp) 
\D_\vp^k \bigg( \cosh \frac{\pi n \vp}{t} \bigg).
\end{align*}
Now using Lemma~\ref{lem:ularge} and Lemma~\ref{lem:ch_gl} (with $v = \pi n/t$) we infer that there exists $M_1 \ge 1$ such that
\begin{align} \nonumber
& \big| \D^N \big[ W_t(\vp - 2\pi n) + W_t(\vp + 2\pi n) \big] \big| \\ \nonumber
& \qquad \lesssim
e^{-\pi^2 n^2/t} \sum_{k=0}^N W_t (\vp) (\pi - \vp)^{-(N-k)} t^{-(N-k)}
e^{\pi n \vp/t} \Big( \frac{n}{t} \Big)^{2k} (\pi - \vp)^{-2k} \\ \label{est3}
& \qquad \lesssim
e^{-\pi^2 n^2/t + \pi n \vp/t} W_t (\vp) \Big( \frac{t}{\pi - \vp} \Big)^{2N} 
\Big( \frac{n}{t} \Big)^{4N}  
\lesssim
e^{-\pi^2 n^2/t + \pi n \vp/t} W_t (\vp) \Big( \frac{n}{t} \Big)^{4N}
\end{align}
holds uniformly in $\vp \in (0,\pi)$, $0 < t \le 1$ and $n \ge 1$ satisfying 
$(\pi - \vp)/t \ge M_1$. 
Summing over $n \ge 2$, we get
\begin{align*}
& \sum_{n=2}^{\8} 
\big| \D^N \big[ W_t(\vp - 2\pi n) + W_t(\vp + 2\pi n) \big]  \big| \\
& \quad \lesssim 
\sum_{n=2}^{\8} e^{-\pi^2 n(n-1)/t } W_t (\vp) \Big( \frac{n}{t} \Big)^{4N} 
\lesssim
W_t (\vp)
\sum_{n=2}^{\8} e^{-\pi^2 n/(2t)}
\lesssim
W_t(\vp) (\pi - \vp)^{-N} t^{-N} e^{-\pi^2/t},
\end{align*}
uniformly in $\vp \in (0,\pi)$ and $0 < t \le 1$ satisfying $(\pi - \vp)/t \ge M_1$.

We now focus on the term with $n=1$ in \eqref{est1}. Using \eqref{est3} we see that
\begin{align*}
\big| \D^N \big[ W_t(\vp - 2\pi) + W_t(\vp+ 2\pi) \big] \big| 
\lesssim
e^{-\pi^2/(2t)} W_t (\vp) t^{-4N}
\lesssim
e^{-\pi^2/(4t)} W_t (\vp) (\pi - \vp)^{-N} t^{-N},
\end{align*}
uniformly in $\vp \in (0,\pi/2]$ and $0 < t \le 1$ satisfying  $(\pi - \vp)/t \ge M_1$.
Therefore, in order to prove \eqref{est1} it is enough to show that
\begin{align} \label{red3}
\big| \D^N \big[ W_t(\vp - 2\pi) + W_t(\vp + 2\pi) \big] \big| 
\lesssim
W_t (\vp) (\pi - \vp)^{-N} t^{-N} e^{-\pi (\pi - \vp)/t},
\end{align}
uniformly in $\vp \in (\pi/2,\pi)$ and $0 < t \le 1$ satisfying 
$(\pi - \vp)/t \ge 1$.

Using \eqref{iden111} and the estimates
\[
| \Phi_{N,j} (\vp \pm 2\pi) |
\lesssim
(\pi - \vp)^{-2N+j}, \qquad \vp \in (\pi/2,\pi), \quad 1 \le j \le N,
\]
which follow from Lemma~\ref{lem:A}, we obtain
\begin{align*}
\big| \D^N W_t(\vp \pm 2\pi) \big|
\lesssim
W_t(\vp \pm 2\pi) (\pi - \vp)^{-N} t^{-N},
\end{align*}
uniformly in $\vp \in (\pi/2,\pi)$ and $t > 0$ satisfying 
$(\pi - \vp)/t \ge 1$. 
Combining this with the relations
\begin{align*}
W_t(\vp + 2\pi) \le W_t(\vp - 2\pi) 
= W_t(\vp) e^{-\pi(\pi - \vp)/t},
\qquad \vp \in (0,\pi), \quad t > 0,
\end{align*}
we get \eqref{red3} and hence also \eqref{est1}.

Next, an application of Lemma~\ref{lem:ularge} shows that there exists $M_2 > 0$ such that 
\begin{align*}
\H^N W_t (\vp)
\simeq
W_t(\vp) (\pi - \vp)^{-N} t^{-N},
\end{align*}
uniformly in $\vp \in (0,\pi)$ and $t>0$ satisfying 
$(\pi - \vp)/t \ge M_2$.
Finally, combining this with \eqref{est1} and the fact that 
$e^{-\pi^2/(4t)} + e^{-\pi (\pi - \vp)/t} \to 0$ as $t \to 0^+$ and 
$(\pi - \vp)/t \to \8$, we see that we can find 
$0 < t_0 \le 1$ and $M_0 \ge M_1 + M_2$ such that 
\begin{align*} 
\H^N \vt_t (\vp)
\simeq
W_t(\vp) (\pi - \vp)^{-N} t^{-N},
\end{align*}
uniformly in $\vp \in (0,\pi)$ and $0 < t \le t_0$ satisfying 
$(\pi - \vp)/t \ge M_0$.
This finishes the proof of Lemma~\ref{lem:thetalar}.
\qed

\subsection{Proof of Lemma \ref{lem:thetasm}} \label{ssec:7}
Notice that for $N = 0$ Lemma~\ref{lem:thetasm} follows in a straightforward way. Therefore we assume $N \ge 1$. 
We first prove that for any $M_0 > 0$ there exists $t_1 > 0$ such that
\begin{align} \label{est2}
\sum_{n=1}^{\8} 
\big| (-\D)^N \big[ W_t(\vp + 2\pi n) + W_t\big(\vp - 2\pi (n+1) \big) \big] \big|
\lesssim
W_t(\vp) t^{-2N} e^{-\pi^2/(4t)},
\end{align}
uniformly in $\vp \in [\pi/2,\pi)$ and $0 < t \le t_1$ satisfying 
$(\pi - \vp)/t \le M_0$.

Almost as in the proof of Lemma~\ref{lem:usmall}, we let $\psi = \pi - \vp \in (0, \pi/2]$ and get
\begin{align*} 
W_t(\vp + 2\pi n) + W_t\big(\vp - 2\pi (n+1) \big)
& =
W_t \big( (2n+1)\pi - \psi \big) + W_t \big( (2n+1)\pi + \psi \big) \\
& =
2e^{-\pi^2 (2n+1)^2/(4t)} W_t(\psi) \cosh \frac{\pi (2n+1) \psi}{2t}
\end{align*}
for $t > 0$ and $n \ge 1$. 
Using Leibniz' rule for $\D$, we get
\begin{align*} 
& (-\D)^N \big[ W_t(\vp + 2\pi n) + W_t\big(\vp - 2\pi (n+1) \big) \big] \\
& \qquad =
2e^{-\pi^2 (2n+1)^2/(4t)} 
\sum_{k=0}^N \binom{N}{k} \D^{N-k} W_t (\psi) 
\D_\psi^k \bigg( \cosh \frac{\pi (2n+1) \psi}{2t} \bigg).
\end{align*}
In view of Corollary~\ref{cor:ularge}, there exists $t_1 > 0$ such that
\begin{align*} 
\big| \D^{N-k} W_t (\psi) \big|
\lesssim
W_t (\psi) t^{-(N-k)}, 
\qquad \psi \in (0, \pi/2], \quad 0 < t \le t_1, \quad 0 \le k \le N.
\end{align*}
Further, using the relation 
\begin{align*} 
W_t (\psi) \simeq W_t (\vp) e^{\pi^2/(4t)}, 
\qquad \psi \in (0, \pi/2], \quad t > 0, \quad \psi/t \le M_0,
\end{align*}
and Lemma~\ref{lem:ch_gl} (with $v = \pi (2n+1)/(2t)$), we arrive at
\begin{align*}
& \big| (-\D)^N \big[ W_t(\vp + 2\pi n) + W_t\big(\vp - 2\pi (n+1) \big) \big]  \big| \\
& \qquad \lesssim
e^{-\pi^2 (2n+1)^2/(4t)} 
\sum_{k=0}^N W_t(\vp) e^{\pi^2/(4t)} t^{-(N-k)} e^{\pi (2n+1) \psi/(2t)}
\Big( \frac{n}{t} \Big)^{2k} \\
& \qquad \lesssim
e^{-\pi^2 (2n+1)^2/(4t) + \pi^2/(4t) + \pi^2 (2n+1)/(2t)} W_t(\vp)
\Big( \frac{n}{t} \Big)^{2N},
\end{align*}
uniformly in $\vp \in [\pi/2,\pi)$, $0 < t \le t_1$ and $n \ge 1$ satisfying 
$(\pi - \vp)/t \le M_0$.
Since $(2n+1)^2 - 1 - 2(2n+1) = 4n^2 - 2 \ge 2n$, $n \ge 1$, we see that the left-hand side of \eqref{est2} is controlled by
\begin{align*}
\sum_{n=1}^{\8}
e^{-\pi^2 n/(2t)} W_t(\vp)
\Big( \frac{n}{t} \Big)^{2N}
\lesssim
W_t(\vp) \sum_{n=1}^{\8} e^{-\pi^2 n/(4t)}
\lesssim
W_t(\vp) t^{-2N} e^{-\pi^2/(4t)},
\end{align*}
uniformly in $\vp \in [\pi/2,\pi)$ and $0 < t \le t_1$ satisfying 
$(\pi - \vp)/t \le M_0$. 
This shows \eqref{est2}. 

Finally, combining \eqref{est2} with Lemma~\ref{lem:usmall} and using the fact that 
$e^{-\pi^2/(4t)} \to 0$ as $t \to 0^+$ we may choose $t_0 > 0$ such that
\begin{align*} 
\H^N \vt_t (\vp)
\simeq
W_t(\vp) t^{-2N},
\end{align*}
uniformly in $\vp \in [\pi/2,\pi)$ and $0 < t \le t_0$ satisfying 
$(\pi - \vp)/t \le M_0$. 
\qed

\vspace{0,5cm}
\section{Proof of Theorem \ref{thm:step2}} \label{sec:st2}

The proof reduces to a thorough analysis of a one-dimensional integral.
Changing the variable of integration in \eqref{reduc1} and using a basic trigonometric identity leads to the formula
\begin{equation*}  
K_t^d(\vp) = \frac{\tilde{c}_d}{(\cos\frac{\vp}2)^{d-2}} \int_{\vp/2}^{\pi-\vp/2} K_{t/4}^{2d-1}(\psi)
	\big( \cos\vp - \cos 2\psi\big)^{(d-3)/2} \sin\psi\, d\psi, \qquad d \ge 2,
\end{equation*}
with $\tilde{c}_d = 2^{-(d-3)/2}c_d$; here $c_d$ is as in \eqref{reduc1}.
Therefore, taking into account Theorem \ref{thm:step1},
in order to prove Theorem \ref{thm:step2} it is enough to show that for any $N \ge 2$
\begin{align*}
& \frac{1}{(\cos\frac{\vp}2)^{N-2}} \int_{\vp/2}^{\pi-\vp/2} \frac{1}{(t+\pi-\psi)^{N-1} \, t^{N-1/2}}
	e^{-\psi^2/t} (\cos\vp - \cos 2\psi)^{(N-3)/2} \sin\psi \, d\psi \\
& \qquad \simeq \frac{1}{(t+\pi-\vp)^{(N-1)/2} \, t^{N/2}} e^{-\vp^2/(4t)},
\end{align*}
uniformly in $\vp \in (0,\pi)$ and $0 < t \le T$ for some fixed $T>0$, say $T=1$.
Moreover, here we can replace the upper limit of integration by $\pi/2$, since the essential contribution
in the last integral comes from integrating over the first half of the interval, that is $(\vp/2,\pi/2)$.
This is seen by reflecting $\psi \mapsto \pi - \psi$ and then using the simple bound
$$
\frac{1}{(t+\psi)^{N-1}} e^{-\pi(\pi-2\psi)/t} \lesssim 1, \qquad \psi \in (0,\pi/2), \quad 0 < t \le 1.
$$

Changing now the variable of integration $\psi = \gamma+\vp/2$ and then replacing the difference of cosines by
a product of sines, we see that it is enough to prove that
\begin{align*}
& \int_{0}^{(\pi-\vp)/2} \frac{1}{(t+\pi-\gamma - \frac{\vp}{2})^{N-1} \, t^{N-1/2}}
	e^{-(\gamma+\vp/2)^2/t} \big[\sin(\gamma+\vp) \sin\gamma \big]^{(N-3)/2} \sin\Big(\gamma+\frac{\vp}{2}\Big) \, d\gamma \\
& \qquad \simeq \Big(\cos\frac{\vp}2\Big)^{N-2}\frac{1}{(t+\pi-\vp)^{(N-1)/2} \, t^{N/2}} e^{-\vp^2/(4t)},
\end{align*}
uniformly in $\vp \in (0,\pi)$ and $0 < t \le 1$.
Since for such $\vp$, $t$ and $\gamma \in (0,(\pi-\vp)/2)$ one has
$\cos(\vp/2) \simeq \pi - \vp$, $t+\pi-\gamma-\vp/2 \simeq 1$, $\sin(\gamma+\vp)\simeq (\pi-\vp)(\gamma+\vp)$,
$\sin\gamma \simeq \gamma$, $\sin(\gamma+\vp/2) \simeq \gamma+\vp$, this task reduces to proving that
$$
\int_0^{(\pi-\vp)/2} \big[ \gamma (\gamma+\vp)\big]^{(N-3)/2} e^{-\gamma(\gamma+\vp)/t} (\gamma+\vp)\, d\gamma
\simeq \big[ t \wedge (\pi-\vp)\big]^{(N-1)/2},
$$
uniformly in the $\vp$ and $t$ in question.

Denoting the last integral by $I$ and changing the variable
$\gamma(\gamma+\vp)/t=s$ we get
\begin{align*}
I & \simeq t^{(N-1)/2} \int_0^{\frac{(\pi-\vp)(\pi+\vp)}{4t}} s^{(N-3)/2} e^{-s}\, ds \\
	& \simeq t^{(N-1)/2}\bigg( 1 \wedge \frac{(\pi-\vp)(\pi+\vp)}{4t} \bigg)^{(N-1)/2}
	\simeq \big[ t \wedge (\pi-\vp)\big]^{(N-1)/2},
\end{align*}
as desired. Theorem \ref{thm:step2} follows.

\end{document}